\documentclass[12pt]{amsart}
\usepackage{amssymb,amsmath,latexsym,amscd}

\usepackage[noadjust]{cite}
\usepackage{bm}
\usepackage{enumerate}
\usepackage{color}

\usepackage%
{hyperref}
\hypersetup{%
  colorlinks,
  bookmarksopen,
  bookmarksnumbered,
  citecolor=blue,
  linkcolor=black,
  urlcolor=blue,
  pdfstartview=FitH,
}

\newtheorem{prop}{Proposition}
\newtheorem{theorem}[prop]{Theorem}

\theoremstyle{definition}
\newtheorem*{remark}{Remark}

\newcommand{\rc}{\mathrm{RC}}
\newcommand{\N}{{\mathbb{N}}}

\def\C{{\mathbb{C}}}
\newcommand{\K}{{\mathcal{K}}}
\newcommand{\M}{{\mathcal{M}}}
\newcommand{\B}{{\mathcal{B}}}
\renewcommand{\H}{{\mathcal{H}}}
\newcommand{\ls}{\leqslant}
\newcommand{\gs}{\geqslant}
\newcommand{\la}{\langle}
\newcommand{\ra}{\rangle}
\newcommand{\hk}{\hookrightarrow}
\renewcommand{\c}{\colon}
\newcommand{\Z}{{\mathcal{S}}}
\newcommand{\ve}{{\varepsilon}}

\title[Grothendieck's inequality in $\Z$]{Grothendieck's inequality\\ in the noncommutative Schwartz space}
\author{Rupert H. Levene}
\address{%
School of Mathematics and Statistics\\
University College Dublin\\
Bel\-field\\Dublin 4\\Ireland}
\email{rupert.levene@ucd.ie}

\author{Krzysztof Piszczek}
\address{Faculty of Mathematics and Comp. Sci.\\
A. Mickiewicz University in Pozna{\'n}\\
Umultowska 87\\61-614 Pozna{\'n}\\Poland}
\email{kpk@amu.edu.pl}

\begin{document}
\begin{abstract}
  In the spirit of Grothendieck's famous inequality from the theory of
  Banach spaces, we study a sequence of inequalities for
  the noncommutative Schwartz space, a Fr\'echet algebra of smooth
  operators. These hold in non-optimal form by a simple nuclearity
  argument. We obtain optimal versions and reformulate the inequalities in several different ways.
\end{abstract}

 \thanks{The research of the second-named author has been supported by the National Center of Science, Poland, grant no. UMO-2013/10/A/ST1/00091.}\thanks{The authors are grateful to the anonymous referee, whose comments greatly improved the presentation of this paper.}
\maketitle

\newcommand{\query}[3][blue]{\marginpar{\raggedright\tiny\color{#1} #3}{\color{#1}#2}}
\renewcommand{\query}[3][]{#2}

\section{Introduction}

The noncommutative Schwartz space~$\Z$ is a weakly amenable m-convex
Fr\'echet algebra whose properties have been investigated in several
recent papers, see e.g.~\cite{TC,PD,KP,KP1}. It is not difficult to see that as a Fr\'echet space,~$\Z$ is nuclear. From this, we can easily deduce the following
analogue of Grothendieck's inequality, which we call
\emph{Grothendieck's inequality in~$\Z$}: there exists a
constant $K>0$ %
so that for any continuous bilinear
form $u\c\Z\times\Z\to\C$ and any $n\in\N$, there exists $k\in\N$ such
that for every $m\in\N$ and any $x_1,\ldots,x_m,y_1,\ldots,y_m\in\Z$,
we have
\begin{equation}
\Big|\sum_{j=1}^mu(x_j,y_j)\Big|\ls K\|u\|_n^* \,\|(x_j)\|_k^{\rc}\,\|(y_j)\|_k^{\rc}
\label{target-inequality}
\end{equation}
The norms on the right hand side %
arise naturally from the
definition of~$\Z$, as explained in Section~\ref{section:2} below. Our
goal in this note is to show that in fact $k=2n+1$ always suffices,
and that this is best possible.

This appears to be the first result concerning Grothendieck's inequality in the category of Fr\'echet algebras; to the best of our knowledge, all previous results along these lines concern Banach spaces (including C${}^*$-algebras, general Banach algebras and operator spaces). For Fr\'echet algebras, Grothendieck's inequality seems to have a specific flavour.  Every Fr\'echet space (and a fortiori, every Fr\'echet algebra) which appears naturally in analysis is nuclear, meaning that all tensor product topologies are equal. Since Grothendieck's inequality can be understood as the equivalence of two tensor products, it seems that we can take inequality~\eqref{target-inequality} for granted. The interesting question that remains is then optimality.

This paper is divided into four sections. In the remainder of this
section, we recall a C${}^*$-algebraic version of Grothendieck's
inequality due to Haagerup, and then review the definition and the basic
properties of~$\Z$ which we require. In Section~2 we explain how
nuclearity gives Grothendieck's inequality in~$\Z$, and we estimate the constants $K$ and $k$. Section~3 then settles the optimality question for~$k$ via a matricial construction.  %
We conclude with a short section containing several reformulations of
the
inequality.%

\subsection{Grothendieck's inequality}

 Pisier's survey article~\cite{GP} is a comprehensive reference for Grothendieck's inequality. This presents many equivalent formulations and applications of this famous result, and recounts its evolution from `commutative'~\cite{AG} to `noncommutative'. %
Of these reformulations and extensions,  Haagerup's noncommutative version most closely resembles~(\ref{target-inequality}), and we state it here for the convenience of the reader.

\begin{theorem}[{\cite{UH}, \cite[Theorem~7.1]{GP}}]
  Let~$A$ and~$B$ be C${}^*$-algebras. For any bounded bilinear
  form~$u\colon A\times B\to \mathbb C$ and any finite sequence $(x_j,y_j)$
  in $A\times B$, we have
  \begin{equation*} \left|\sum u(x_j,y_j)\right| \ls 2\|u\|\,\|(x_j)\|^{\rc}\,\|(y_j)\|^{\rc}
  \end{equation*}
  where $\|(x_j)\|^{\rc} := \max\big\{
      \big\|\sum x_j^*x_j\big\|^{\frac12}
      ,
      \big\|\sum x_jx_j^*\big\|^{\frac12}
    \big\}$.

\end{theorem}

\subsection{The noncommutative Schwartz space}

Let
\[s=\Big\{\xi=(\xi_j)_{j\in\N}\in\C^{\N}\colon |\xi|_n:=\Big(\sum_{j=1}^{+\infty}|\xi_j|^2j^{2n}\Big)^{\frac12}<+\infty\,\,\text{for all}\,\,n\in\N\Big\}\]
denote the so-called space of rapidly decreasing sequences. This space becomes Fr\'echet when endowed with the sequence $(|\cdot|_n)_{n\in\N}$ of norms defined above.
The basis $(U_n)_{n\in\N}$ of zero neighbourhoods of $s$ is defined by $U_n:=\{\xi\in s\colon |\xi|_n\ls1\}$. The topological dual~$s'$ of $s$ is the so-called space of slowly increasing sequences, namely
\[\Big\{\eta=(\eta_j)_{j\in\N}\in\C^{\N}\colon |\eta|_n':=\Big(\sum_{j=1}^{+\infty}|\eta_j|^2j^{-2n}\Big)^{\frac12}<+\infty\,\,\text{for some}\,\,n\in\N\Big\}\]
where the duality pairing is given by
$
\langle\xi,\eta\rangle:=\sum_{j\in\N}\xi_j\overline{\eta_j}$ for $\xi\in s$, $\eta\in s'$.

The \textit{noncommutative Schwartz space} $\Z$ is the Fr\'echet space $L(s',s)$ of all  continuous linear operators from $s'$ into $s$, endowed with the topology of uniform convergence on bounded sets. %
The formal identity map $\iota\c s\hk s'$ is a continuous embedding and defines a product on $\Z$ by
$xy:=x\circ\iota\circ y$ for $x,y\in\Z$.
There is also a natural involution on $\Z$ given by
$\langle x^*\xi,\eta\rangle:=\langle\xi,x\eta\rangle$ for $x\in\Z$, $\xi,\eta\in s'$.
With these operations, $\Z$ becomes an m-convex Fr\'echet $*$-algebra. The inclusion map $\Z\hk\K(\ell_2)$ is continuous, and in fact it is a spectrum-preserving $*$-homomorphism~\cite{PD}. %
Moreover~\cite[Proposition~3]{KP}, an element $x\in \Z$ is positive (i.e., $x=y^*y$ for some $y\in \Z$), if and only if the spectrum of $x$ is contained in~$[0,+\infty)$, or equivalently $\langle x\xi,\xi\rangle\geq 0$ for all $\xi\in s'$.
On the other hand, by~\cite[Cor. 2.4]{PD} %
and~\cite[Theorems~8.2, 8.3]{MF}, the topology of %
$\Z$ %
cannot be given by a sequence of C$^*$-norms. This %
causes some technical inconvenience (e.g.~there is no bounded
approximate identity in~$\Z$) meaning we cannot apply C$^*$-algebraic
techniques directly. %

\section{The inequality}\label{section:2}

Let $(\|\cdot\|_n)_{n\in\N}$ be a non-decreasing sequence  of norms which gives the topology of~$\Z$. For  $u\c\Z\times\Z\to\C$ a continuous bilinear form, we write
\[\|u\|^*_n:=\sup\{|u(x,y)|\colon x,y\in\mathcal{U}_n\}\]
where $\mathcal{U}_n=\{x\in\Z\c\,\|x\|_n\ls1\}$; similarly, for a functional $\phi\in \Z'$, we write
\[ \|\phi\|^*_n:=\sup\{ |\phi(x)|\c\, x\in\mathcal{U}_n\}.\] %
Following
Pisier~\cite[p.~316]{GP2}, for $k\in\N$ and $x_1,x_2,\dots,x_m\in \Z$, we write
\[ \|(x_j)\|_k^{\rc}=\max\Big\{\Big\|\sum_{j=1}^mx_j^*x_j\Big\|_{k}^{\frac12},\Big\|\sum_{j=1}^mx_jx_j^*\Big\|_{k}^{\frac12}\Big\}.\]
Relative to our choice of norms $\|\cdot\|_n$, we
have now defined each term in our hoped-for
inequality~(\ref{target-inequality}). We will now reformulate it using
tensor products.

For C$^*$-algebras, such a reformulation is standard. Indeed, by~\cite[Theorem~1.1]{UH} (formulated along the lines of~\cite[Theorem~2.1]{KS}), Haagerup's noncommutative Grothendieck inequality entails the existence of a $K>0$ such that for any C$^*$-algebras $A,B$ and $z$ in the algebraic tensor product $A\otimes B$, we have $\|z\|_{\pi}\ls K\|z\|_{ah}$
where $\|\cdot\|_{\pi}$ is the projective tensor norm and
$\|\cdot\|_{ah}$ is the absolute Haagerup tensor norm~\cite[p.~164]{KS} on~$A\otimes B$, given by
\[ \|z\|_{ah}=\inf\Big\|\sum_{j=1}^m|x_j|^2\Big\|^{\frac12}\Big\|\sum_{j=1}^m|y_j|^2\Big\|^{\frac12}.\]
Here $|x|=\big(\tfrac12(x^*x+xx^*)\big)^{\frac12}$ for $x$ an element of a C$^*$-algebra, and the infimum is taken over all representations $z=\sum_{j=1}^{m} x_j\otimes y_j$ where $(x_j,y_j)\in A\times B$.

We proceed similarly for $\Z$. For $x\in\Z$, let
$|x|^2=\tfrac12(x^*x+xx^*)\in \Z$
and consider the sequence
of \textit{absolute Haagerup tensor norms} %
$(\|\cdot\|_{ah,n})_{n\in\N}$ on the algebraic tensor product $\Z\otimes\Z$ given by
\[\|z\|_{ah,n}:=\inf\Big\|\sum_{j=1}^m|x_j|^2\Big\|_n^{\frac12}\Big\|\sum_{j=1}^m|y_j|^2\Big\|_n^{\frac12}\]
where the infimum runs over all ways to represent $z=\sum_{j=1}^mx_j\otimes y_j$ in $\Z\otimes\Z$. As  usual, we write $(\|\cdot\|_{\pi,n})_{n\in\N}$ for the sequence of projective tensor norms on $\Z\otimes \Z$.

Just as in the C$^*$-algebra case, inequality \eqref{target-inequality} will follow once we show that the sequences of projective and absolute Haagerup tensor norms are equivalent on $\Z\otimes\Z$. In fact, the equivalence of these norms follows immediately from the nuclearity of~$\Z$ (see~\cite[Theorem~28.15]{MV} and~\cite[Ch. 21, \S2, Theorem~1]{J} for details). On the other hand, the optimal values of $k$ and~$K$ (depending on~$n$ and our choice of norms~$(\|\cdot\|_n)_{n\in\N}$) for which~\eqref{target-inequality} hold are not given by such general considerations. These optimal parameters will be denoted by $\kappa(n):=k_{best}$ and $K_n:=K_{best}$.

Henceforth, we focus only on the sequence of norms $(\|\cdot\|_n)_{n\in\N}$ where
\[\|x\|_n:=\sup\{|x\xi|_n\colon \xi\in U^{\circ}_n\},\qquad n\in\N,\ x\in\Z\]
and $U_n^{\circ}=\{\xi\in s'\c\,\,|\xi|_n'\ls1\}$. In other words, $\|x\|_n$ is the norm of~$x\in \Z$, considered as a Hilbert space operator from $H_n':=\ell_2((j^{-n})_j)$ to $H_n:=\ell_2((j^n)_j)$.
This sequence does indeed induce the topology of~$\Z$.
In this context, we will estimate~$K_n$ and compute the exact values of~$\kappa(n)$.

We start with the following result, which can be compared with~\cite[Lemma~8]{KP}. To fix some useful notation, for $n\in\mathbb N$ we define an infinite diagonal matrix $d_n:=\operatorname{diag}(1^n,2^n,3^n,4^n,\dots)$\label{dn-def} which we consider as an isometry $d_n\c\ell_2\to H_n' $ and simultaneously as an isometry $d_n\c H_n\to\ell_2$.

\goodbreak
\begin{prop}
\label{inequalities}
Let~$n\in\mathbb N$. We have

\begin{enumerate}
\item[(i)]
$\|x\|_n=\sup\{\la x\xi,\xi\ra\colon \xi\in U^\circ_n\}$
for every positive $x\in\Z$;

\item[(ii)] $\|x\|_n^2\ls\|x^2\|_{2n}$ for every self-adjoint $x\in\Z$; and

\item[(iii)] $\|x\|_n^2\ls\|x^*x\|_{2n}^{\frac12}\|xx^*\|_{2n}^{\frac12}$ for every $x\in\Z$.
\end{enumerate}
Moreover, inequalities~(ii) and~(iii) are sharp.
\end{prop}

\begin{proof}  (i) Observe that $\|x\|_n=\|d_nxd_n\|_{\B(\ell_2)}$. Furthermore, since~$x$ is positive, $d_nxd_n$ is positive %
and we have
\begin{align*}
\|x\|_n & =\|d_nxd_n\|_{\B(\ell_2)}
 =\sup\{\la xd_n\xi,d_n\xi\ra\colon |\xi|_{\ell_2}\ls1\} \\
 &=\sup\{\la x\xi,\xi\ra\colon |\xi|_n'\ls1\}.
\end{align*}

(ii)  For $x$ self-adjoint, we have \[\|x^2\|_{2n}=\|d_{2n}x^2d_{2n}\|_{\B(\ell_2)}=\|d_{2n}x\|^2_{\B(\ell_2)},\]
and by~\cite[Proposition~II.1.4.2]{BB},
\[\|x\|_n=\|d_nxd_n\|_{\B(\ell_2)}=\nu(d_nxd_n)=\nu(d_{2n}x)\ls\|d_{2n}x\|_{\B(\ell_2)},\]
where $\nu(\cdot)$ denotes the spectral radius. This gives the desired inequality.

(iii) Since $\Z\hk\B(\ell_2)$, any $x\in \Z$ is also a Hilbert space operator, and the block-matrix operator
$\left[\begin{smallmatrix}
(xx^*)^{1/2} & x \\
x^* & (x^*x)^{1/2}
\end{smallmatrix}\right]$
is positive in $\B(\ell_2\oplus\ell_2)$ (see e.g.%
~\cite[p.~117]{VP}). Equivalently,
\begin{equation}
|\la x\xi,\eta\ra|^2\ls\la(xx^*)^{1/2}\eta,\eta\ra\la(x^*x)^{1/2}\xi,\xi\ra\qquad\forall\,\xi,\eta\in\ell_2.
\label{positive-ineq1}
\end{equation}
For $m\in\N$, let us write $p_m:=\left[\begin{smallmatrix}
I_m & 0 \\0 & 0
\end{smallmatrix}\right]
$
where $I_m\in\M_m$ is the identity matrix. Now fix $n\in\N$ and choose $\xi,\eta\in H_n'$. Then $p_m\xi,p_m\eta\in\ell_2$ for all $m\in\N$ and \eqref{positive-ineq1} gives
\[|\la p_mxp_m\xi,\eta\ra|^2\ls\la p_m(xx^*)^{1/2}p_m\eta,\eta\ra\la p_m(x^*x)^{1/2}p_m\xi,\xi\ra.\]
Since $(p_m)_{m\in\N}$ is an approximate identity in $\Z$ (see~\cite[Proposition~2]{KP}), we obtain
\[|\la x\xi,\eta\ra|^2\ls\la(xx^*)^{1/2}\eta,\eta\ra\la(x^*x)^{1/2}\xi,\xi\ra.\]
Taking the supremum over all $\xi,\eta$ in the unit ball of $H_n'$ we get
\[\|x\|_n^2\ls\|(xx^*)^{1/2}\|_n\|(x^*x)^{1/2}\|_n.\]
Applying (ii) to the positive operators $(xx^*)^{1/2}$ and $(x^*x)^{1/2}$ we conclude that
$\|x\|_n^2
 \ls\|x^*x\|_{2n}^{\frac12}\|xx^*\|_{2n}^{\frac12}$.

For sharpness, observe that if~$x$ is a diagonal rank one matrix unit then we have equality in both~(ii) and~(iii).
\end{proof}

\begin{prop}
\label{prop-fourth-root}
For any $n,m\in\N$ and $x_1,\ldots,x_m,y_1,\ldots,y_m\in\Z$, we have
\begin{multline*}
\sum_{j=1}^m\|x_j\|_n\|y_j\|_n \\
\ls\frac{\pi^2}{6}\Big\|\sum_{j=1}^mx_j^*x_j\Big\|_{2n+1}^{\frac14}
\Big\|\sum_{j=1}^mx_jx_j^*\Big\|_{2n+1}^{\frac14}
\Big\|\sum_{j=1}^my_j^*y_j\Big\|_{2n+1}^{\frac14}
\Big\|\sum_{j=1}^my_jy_j^*\Big\|_{2n+1}^{\frac14}.
\end{multline*}
\end{prop}

\begin{proof}
Let $C:=\frac{\pi^2}{6}$ and let~$p\in \N$. We claim that
\[
\sum_{k=1}^m\|x_k^*x_k\|_p\ls C\Big\|\sum_{k=1}^mx_k^*x_k\Big\|_{p+1}.
\]
By the Cauchy--Schwarz inequality and Proposition~\ref{inequalities}(iii) this will then imply the desired inequality. To establish the claim, %
let $\xi^1,\ldots,\xi^m\in U_p^{\circ}$ and let us write $(e_{j})_{j\in\N}$ for the standard basis vectors in~$\ell^2$. We have
\[\sum_{k=1}^m\la x_k^*x_k\xi^k,\xi^k\ra=\sum_{k=1}^m\sum_{i,j=1}^{+\infty}\la x_k^*x_ke_j,e_i\ra (ij)^p\overline{\xi_i^k}i^{-p}\xi_j^kj^{-p}.\]
Applying the Cauchy--Schwarz inequality to summation over $i,j\in\N$ gives
\[\sum_{k=1}^m\la x_k^*x_k\xi^k,\xi^k\ra\ls\sum_{k=1}^m\Bigl(\sum_{i,j=1}^{+\infty}|\la x_k^*x_ke_j,e_i\ra|^2(ij)^{2p}\Bigr)^{\frac12}.\]
 Since $x^*x$ is positive for any $x\in\Z$, and for positive operators $y\in\Z$ we have $|y_{ij}|^2\ls y_{ii}y_{jj}$ (where $y_{ij}:=\langle ye_j,e_i\rangle$), this implies that
\begin{align*}
\sum_{k=1}^m\la x_k^*x_k\xi^k,\xi^k\ra & \ls\sum_{j=1}^{+\infty}\Big\la\Bigl(\sum_{k=1}^mx_k^*x_k\Bigr)j^pe_j,j^pe_j\Big\ra  \\
& \ls\sum_{j=1}^{+\infty}j^{-2}\sup_{i\in\N}\Big\la\Bigl(\sum_{k=1}^mx_k^*x_k\Bigr)i^{p+1}e_i,i^{p+1}e_i\Big\ra  \\
& \ls C\Big\|\sum_{k=1}^mx_k^*x_k\Big\|_{p+1}.
\end{align*}

By Proposition~\ref{inequalities}(i), for any $\ve>0$ there are $\xi^1,\ldots,\xi^m\in U_p^{\circ}$ with
\[
\sum_{k=1}^m\|x_k^*x_k\|_p<\sum_{k=1}^m\la x_k^*x_k\xi^k,\xi^k\ra+\ve<C\Big\|\sum_{k=1}^mx_k^*x_k\Big\|_{p+1}+\ve.\]
Taking the infimum over $\ve>0$ yields the claim.
\end{proof}

\goodbreak
As a straightforward consequence of Proposition~\ref{prop-fourth-root}, we obtain:

\begin{theorem}[Grothendieck's inequality in~$\Z$]
\label{gt-s}
There is a constant $K\ls\frac{\pi^2}{6}$ such that $\|z\|_{\pi,n}\ls2K\|z\|_{ah,2n+1}$ for any $n\in\N$ and $z\in\Z\otimes\Z$. Moreover,
every continuous bilinear form $u\c\Z\times\Z\to\C$ satisfies inequality~(\ref{target-inequality}) with $k=2n+1$, for any $n,m\in \N$ and any $x_1,\ldots,x_m,y_1,\ldots,y_m\in\Z$. %
In particular, taking $u(x,y):=\phi(x)\phi(y)$ where $\phi\in\Z'$, we obtain
\begin{equation}
\sum_{j=1}^m|\phi(x_j)|^2\ls K(\|\phi\|_n^*\,\|(x_j)\|^{\rc}_{2n+1})^2%
\label{gt-functional}
\end{equation}

\end{theorem}

\begin{remark}
  This shows that
  $\kappa(n)\ls2n+1$. On the other hand, it is easy to show that
  $\kappa(n)>2n-1$. Indeed, if not, then~\eqref{gt-functional} would hold with
   $2n+1$ replaced by some ${\ell}\ls 2n-1$. Take $m\in\N$, define
  $\xi_m:=\sum_{j=1}^mj^ne_j$ and $\phi_m\in\Z'$ by $\phi_m(x):=\la
  x\xi_m,\xi_m\ra$. Then for $x_j:=e_{jj},\,j=1,\ldots,m$ we get $\|(x_j)\|^{\rc}_\ell=m^\ell$
  and
  $\|\phi_m\|_n^*=m$. On the other hand, $\sum_{j=1}^m|\phi_m(x_j)|^2$
  is equivalent (up to a constant) to $m^{4n+1}$. Therefore
  \eqref{gt-functional} takes the form $m^{4n+1}\ls Cm^{2\ell+2}$ for some constant $C$ (independent of
  $m$). Letting $m$ tend to infinity, we obtain ${\ell}\gs2n-\frac12$,
  a contradiction.  Hence $\kappa(n)\in \{2n,2n+1\}$.
\end{remark}

\section{Optimality}

We will now show that $\kappa(n)=2n+1$.  For this, we will use the tensor product formulation, noting that
\[\kappa(n)=\min\Big\{k\in\N\colon \sup\Big\{\frac{\|z\|_{\pi,n}}{\|z\|_{ah,k}}\colon z\in\Z\otimes\Z,\,z\neq0\Big\}<\infty\Big\}.\]
Recall the diagonal operator $d_n$ defined on page~\pageref{dn-def}
above. Since every $x\in\Z$ is an operator on $\ell_2$ via the
canonical inclusions $\ell_2\stackrel{d_n}\hk H_n'\xrightarrow{x}H_n\stackrel{d_n}\hk\ell_2$,
it is clear that if $x\in\Z$, then $d_nx$ and $xd_n$ are both operators on $\ell_2$. This leads to the following observation.

\begin{prop}
\label{elemntary-equalities}
    If $z=\sum_{j=1}^kx_j\otimes y_j\in\Z\otimes\Z$, then
    \[\|z\|_{\pi,n}=\Big\|\displaystyle\sum_{j=1}^kd_nx_jd_n\otimes d_ny_jd_n\Big\|_{\pi}.\]
\end{prop}

\begin{proof}
Write \[\Delta_nz=\sum_{j=1}^kd_nx_jd_n\otimes d_ny_jd_n\in\B(\ell_2)\otimes\B(\ell_2).\] If $\Delta_nz=\sum_{l=1}^ma_l\otimes b_l\in\B(\ell_2)\otimes\B(\ell_2)$ and $\sum_{l=1}^m\|a_l\|\|b_l\|<\|\Delta_nz\|_{\pi}+\ve$ for some $\ve>0$, then $z=\sum_{l=1}^md_n^{-1}a_ld_n^{-1}\otimes d_n^{-1}b_ld_n^{-1}$ and
\[
\|z\|_{\pi,n} \ls\sum_{l=1}^m\|d_n^{-1}a_ld_n^{-1}\|_n\|d_n^{-1}b_ld_n^{-1}\|_n
 =\sum_{l=1}^m\|a_l\|\|b_l\|<\|\Delta_n z\|_{\pi}+\ve.
\]
This gives $\|z\|_{\pi,n}\ls\|\Delta_nz\|_{\pi}$. The reverse inequality is proved similarly.
\end{proof}

We also need the following well-known fact.

\begin{prop}
\label{self-adjoint-haagerup}
If $\H$ is a Hilbert space and $x_1,\ldots,x_m\in\B(\H)$, then
\[\Big\|\sum_{j=1}^mx_j\otimes x_j^*\Big\|_h=\Big\|\sum_{j=1}^mx_jx_j^*\Big\|.\]
\end{prop}

\begin{proof}
By~\cite[Theorem~4.3]{RRS}, the Haagerup norm on the left hand side is equal to the completely bounded norm of the map on~$\B(\H)$ given by~$a\mapsto \sum_{j=1}^m x_jax_j^*$, which is completely positive, so attains its completely bounded norm at the identity operator.
\end{proof}

\begin{theorem}%
For every $n\in\N$, we have $\kappa(n)=2n+1$.
\end{theorem}

\begin{proof}
  By Theorem~\ref{gt-s}, it only remains to show that $\kappa(n)>2n$. %
  Choose $k_n\in\N$ sufficiently large that
  $k+1\ls2^k(k^{\frac{1}{4n}}-1)$ for all $k\gs k_n$.
  This inequality ensures that for every $k\gs k_n$, if we define
  \[i_1=2^k,\quad i_{k+1}=\lfloor k^{\frac{1}{4n}}2^k\rfloor,\quad
  i_j=i_{k+1}+j-(k+1),\ 2\ls j\ls k,\] then
  $i_1<i_2<\dots<i_{k+1}$. %
  Denote by $(e_{ij})_{i,j\in\N}$ the standard matrix units, and for
  $j=2,\ldots,k+1$, consider the self-adjoint operators
  \[x_j:=e_{i_1,i_j}+e_{i_j,i_1}\in\mathcal{M}_{i_{k+1}}\subset\Z.\]
  Let
  \[z_k:=\sum_{j=2}^{k+1}x_j\otimes x_j.\]  Since
  $d_nx_jd_n=i_1^ni_j^n(e_{i_1,i_j}+e_{i_j,i_1})$ and 
  $(d_nx_jd_n)^2=i_1^{2n}i_j^{2n}(e_{i_1,i_1}+e_{i_j,i_j})$,
  by Propositions~\ref{elemntary-equalities} and~\ref{self-adjoint-haagerup} we obtain
  \begin{align*}
    \|z_k\|_{\pi,n} & =\Big\|\sum_{j=2}^{k+1}d_nx_jd_n\otimes d_nx_jd_n\Big\|_{\pi}  \\
    & \gs\Big\|\sum_{j=2}^{k+1}d_nx_jd_n\otimes d_nx_jd_n\Big\|_h
    =\Big\|\sum_{j=2}^{k+1}(d_nx_jd_n)^2\Big\|
    =i_1^{2n}\sum_{j=2}^{k+1}i_j^{2n}.
  \end{align*}

  On the other
  hand, \[|x_j|^2=x_j^2=e_{i_1,i_1}+e_{i_j,i_j}\quad\text{and}\quad
  \sum_{j=2}^{k+1}d_{2n}x_j^2d_{2n}=i_1^{4n}k
  e_{i_1,i_1}+\sum_{j=2}^{k+1}i_j^{4n}e_{i_j,i_j}.\]
  Therefore %
  \begin{align*}
    \|z_k\|_{ah,2n} &\ls\Big\|\sum_{j=2}^{k+1}|x_j|^2\Big\|_{2n}
    =\Big\|\sum_{j=2}^{k+1}d_{2n}x_j^2d_{2n}\Big\|
    =\max\{i_1^{4n}k,i_{k+1}^{4n}\}  \\
    & \ls i_1^{4n}k+i_{k+1}^{4n}.
  \end{align*}
  Hence
  \begin{align*}
    \frac{\|z_k\|_{\pi,n}}{\|z_k\|_{ah,2n}}>\frac{i_1^{2n}\sum_{j=2}^{k+1}
      {i_j^{2n}}}{i_1^{4n}k+i_{k+1}^{4n}}
    >
    \frac{i_1^{-2n}i_2^{2n}}{1+k^{-1}i_1^{-4n}i_{k+1}^{4n}}\to \infty \text{ as $k\to \infty$},
  \end{align*}
  by our choice of $i_1,\dots,i_{k+1}$. So $\kappa(n)>2n$ as required.
\end{proof}

\section{Reformulations of the inequality}

Here we give several different ways of stating our inequality; in each case, an analogous result may be found in~\cite{GP}. The methods here are fairly standard, so full proofs are often omitted. Throughout, we write $K=\sup_{n\in\N}K_n\ls \pi^2/6$.

\subsection{Grothendieck's inequality with states}

Given $\xi\in U_n^{\circ}$, let
$\phi_{\xi}\in \Z'$ be given by $\phi_\xi(x)=\la
x\xi,\xi\ra$, $x\in \Z$. We call an element of the closed convex hull
of $\{ \phi_\xi\colon \xi\in U_n^\circ\}$ an %
\textit{$n$-state} on~$\Z$. %
Note that by Proposition~\ref{inequalities}(i), for any positive element $x\in \Z$ we have
$\|x\|_n=%
\sup\{\phi(x)\colon \phi\in V_n\}$,
where $V_n\subseteq \Z'$ is the set of all $n$-states on~$\Z$.
The next result may be deduced from Theorem~\ref{gt-s} by closely following the Hahn--Banach Separation argument of~\cite[\S23]{GP}.

\begin{theorem}
\label{gt-s-states}
For any continuous bilinear form $u\c\Z\times\Z\to\C$ and $n\in\N$, there are $(2n+1)$-states $\phi_1,\phi_2,\psi_1,\psi_2$ on $\Z$ with
\[|u(x,y)|\ls K\|u\|_n^*\bigl(\phi_1(x^*x)+\phi_2(xx^*)\bigr)^{\frac12}\bigl(\psi_1(y^*y)+\psi_2(yy^*)\bigr)^{\frac12}\] for all $x,y\in\Z$.
\end{theorem}

\subsection{`Little' Grothendieck inequality}

As a consequence we obtain the following `little' Grothendieck inequality in~$\Z$. Recall that if $T\c X\to Y$ is a linear %
 map between Fr\'echet spaces, then $\|T\|_{n,k}:=\sup\{\|Tx\|_k\colon \|x\|_n\ls1\}$.

\begin{theorem}
For any Fr\'echet-Hilbert space~$H$, if $u_1,u_2\c\Z\to H$ are continuous linear maps, $k,m,n\in\N$ and $x_1,\ldots,x_m,y_1,\ldots,y_m\in\Z$, then
\begin{equation*}
\Big|\sum_{j=1}^m\la u_1(x_j),u_2(y_j)\ra_k\Big|\ls K\|u_1\|_{n,k}\,\|u_2\|_{n,k}\,\|(x_j)\|_{2n+1}^{\rc}\,\|(y_j)\|_{2n+1}^{\rc}.%
\end{equation*}
Equivalently, for any $k,n\in\N$ there are $(2n+1)$-states $\phi_1,\phi_2,\psi_1,\psi_2$ such that for all $x,y\in\Z$ we have
\begin{multline*}|\la u_1(x),u_2(y)\ra_k|\ls K\|u_1\|_{n,k}\|u_2\|_{n,k}\\
\times \bigl(\phi_1(x^*x)+\phi_2(xx^*)\bigr)^{\frac12}\bigl(\psi_1(y^*y)+\psi_2(yy^*)\bigr)^{\frac12}.
\end{multline*}
\end{theorem}

\begin{proof}
Apply Theorems~\ref{gt-s} and \ref{gt-s-states} to $u_k(x,y):=\la u_1(x),u_2(y)\ra_k$ for $k\in\N$.
\end{proof}

Using the same argument as in the proof of Theorem~\ref{gt-s-states} we can obtain an equivalent version of the `little' Grothendieck inequality.

\begin{theorem}
For any Fr\'echet-Hilbert space $H$, if $u\c\Z\to H$ is a continuous linear map and $k,n\in\N$, then there exist $(2n+1)$-states $\phi_1,\phi_2$ on $\Z$ such that for all $x\in\Z$ we have
\[\|ux\|_k\ls\sqrt{K}\|u\|_{n,k}\big(\phi_1(x^*x)+\phi_2(xx^*)\big)^{\frac12}.\]
\end{theorem}

\renewcommand{\bibliofont}{\normalsize\baselineskip=17pt}

\bibliographystyle{plain}

\bibliography{bibliography}

\end{document}